\documentclass{article}%

\usepackage{amsmath,amssymb,amsfonts}%
\usepackage{theorem}%
\usepackage{color}%
\usepackage{graphicx}%
\usepackage{hyperref}%
\usepackage{caption}
\usepackage{subcaption}

\theoremstyle{change}%
\newtheorem{definition}{Definition:}[section]%
\newtheorem{proposition}[definition]{Proposition:}%
\newtheorem{theorem}[definition]{Theorem:}%
\newtheorem{lemma}[definition]{Lemma:}%
{\theorembodyfont{\rmfamily}\newtheorem{remark}[definition]{Remark:}}%

\newenvironment{proof}
{{\bf Proof:}}
{\qquad \hspace*{\fill} $\Box$}%

\newcommand{\fa}{\mathfrak{a}}%
\newcommand{\fg}{\mathfrak{g}}%
\newcommand{\fk}{\mathfrak{k}}%
\newcommand{\fn}{\mathfrak{n}}%
\newcommand{\fp}{\mathfrak{p}}%
\newcommand{\fs}{\mathfrak{s}}%
\newcommand{\fz}{\mathfrak{z}}%
\newcommand{\Ad}{\operatorname{Ad}}%
\newcommand{\ad}{\operatorname{ad}}%
\newcommand{\tr}{\operatorname{tr}}%

\newcommand{\inner}{\operatorname{int}}%
\newcommand{\cl}{\operatorname{cl}}%
\newcommand{\rme}{\mathrm{e}}%
\newcommand{\AC}{\mathcal{A}}%
\newcommand{\CC}{\mathcal{C}}%
\newcommand{\DC}{\mathcal{D}}%
\newcommand{\LC}{\mathcal{L}}%
\newcommand{\OC}{\mathcal{O}}%
\newcommand{\UC}{\mathcal{U}}%
\newcommand{\XC}{\mathcal{X}}%
\newcommand{\R}{\mathbb{R}}%
\newcommand{\Z}{\mathbb{Z}}%

\begin{document}

\title{Control sets of linear systems on semi-simple Lie groups}
\author{V\'{\i}ctor Ayala%
\thanks{
Supported by Proyecto Fondecyt n$%
%TCIMACRO{\U{b0}}%
%BeginExpansion
{{}^\circ}%
%EndExpansion
$ 1150292. Conicyt, Chile.} \\
Instituto de Alta Investigaci\'{o}n\\
Universidad de Tarapac\'{a}, Arica, Chile and\\
Departamento de Matem\'{a}ticas\\
Universidad Cat\'{o}lica del Norte, Antofagasta, Chile
 \and Adriano Da Silva\thanks{ Supported by Fapesp grant $n^{o}$ 2016/11135-2 and 2018/10696-6} \\
Instituto de Matem\'{a}tica\\
Universidade Estadual de Campinas, Brazil
\and Philippe Jouan \\ 
Laboratoire de Math\'{e}matiques Rapha\"{e}l Salem\\
CNRS UMR 6085\\
Universit\'{e} de Rouen, France
\and Guilherme Zsigmond\thanks{
Supported by Capes grant BEX 1041-14-2} \\
Departamento de Matem\'{a}ticas\\
Universidad Cat\'{o}lica del Norte, Antofagasta, Chile and\\
Laboratoire de Math\'{e}matiques Rapha\"{e}l Salem\\
CNRS UMR 6085\\
Universit\'{e} de Rouen, France
}
\date{\today }
\maketitle

\begin{abstract}
In this paper we study the main properties of control sets with nonempty interior of linear control systems on semisimple Lie groups. We show that, unlike the solvable case, linear control systems on semisimple Lie groups may have more than one control set with nonempty interior and that they are contained in right translations of the one around the identity.
\end{abstract}

\section{Introduction}

	Linear control systems on Lie groups appear as a natural extension of linear control systems on Euclidean spaces. Several works addressing the main issues in control theory for such systems, such as controllability, observability and optimization appeared over the years. In \cite{JPh1} P.	Jouan showed that such generalization is also important for the classification of general control systems on abstract connected manifolds. 

	On the other hand it is meaningful to deal with restricted inputs because that corresponds to the practical case. However many systems are not controllable for restricted inputs. It is therefore natural to study the maximal regions of controllability, i.e. the control sets (see Section 2.1 for the definition). In the paper at hand we analyze the control sets of linear control systems on semisimple Lie groups. In \cite{GzDaAv}, several topological properties of such sets were proven in the solvable case. By using the close relationship between the dynamics of the drift and the behaviour of the control system (see \cite{DSAy} and \cite{DS}) the authors were able to prove boundedness and uniqueness of the control sets. The richness of the geometry of semisimple Lie groups modifies completely the picture. Indeed more than one control set with nonempty interior may exist. They are contained in right translations of the existent control set around the identity. Moreover, the existence of an {\it invariant control set} implies global controllability, showing how semisimplicity strongly influences the behaviour of the control system.

The paper is structured as follows: Section 2 introduces the main definitions and principal properties concerning control systems, control sets, linear vector fields and linear control systems on Lie groups. Since our work is devoted to semisimple Lie groups, we also provide in Section 2 a small subsection about semisimple theory in order to make the paper self-contained. Section 3 contains the main results concerning control sets with nonempty interior of a linear control system on a connected semisimple Lie group. We show that all the possible control sets with nonempty interior of the system are contained in the right translations of the control set around the identity. The particular case where the drift of the system has trivial nilpotent part, these right translations are precisely the control sets of the system. In this section we also show that for linear control systems on semisimple Lie groups, the only possible invariant control set is the whole group. Section 4 is devoted to illustrating the paper with an example in $\mathrm{Sl}(2)$.

\bigskip

{\bf Notations:} Let $G$ be a connected Lie group. We denote by $e$ the identity element of $G$. For any element $g\in G$, the maps $L_g$ and $R_g$ stand for the left and right translations in $G$, respectively. By $C_g=L_g\circ R_{g^{-1}}=R_{g^{-1}}\circ L_g$ we denote the conjugation of $G$. By $\mathrm{Aut}(G)$ we denote the group of automorphisms of $G$. If $(\varphi_t)_{t\in\R}\subset\mathrm{Aut}(G)$ is a 1-parameter subgroup, its orbit from $g$ is the subset $\OC(g, \varphi)=\{\varphi_t(g), \;t\in\R\}$. We say that a subset $B\subset G$ is $\varphi$-invariant if $\varphi_t(B)\subset B$ for any $t\in\R$.

\section{Preliminaries}

	\subsection{Control systems}

	Let $M$ be a $d$-dimensional smooth manifold. A {\it control system} on $M$
	is a family of ordinary differential equations
	\begin{equation}
	\label{controlsystem}
	\dot{x}(t)=f(x(t), u(t)), \; \; \;u\in \mathcal{U},
	\end{equation}
	where $f: M\times \mathbb{R}^{m}\rightarrow TM$ is a smooth map
	and 
	$$\mathcal{U}:=\{u:\R\rightarrow\R^m; \;u\;\mbox{ is measurable with }\;u(t)\in \Omega\;\mbox{ a.e.}\},$$
	is the set of the {\it admissible control functions}, with $\Omega$ a bounded subset of $\mathbb{R}^{m}$ such that $0\in\inner\Omega$. For any $x\in M$ and
	$u\in \mathcal{U}$ we denote by $\phi(t, x, u)$ the unique solution of
	(\ref{controlsystem}) with initial value $x=\phi(0, x, u)$. We use $\phi_{t, u}$ to denote the diffeomorphism $x\in M\mapsto \phi(t, x, u)\in M$. 	
	Given $u_1, u_2\in\UC$ and $t_1, t_2>0$ we have that 
	$$\phi(t_1, \phi(t_2, x, u_2), u_1)=\phi(t_1+t_2, x, u)$$
	where $u=u_1*u_2\in\UC$ is the {\it concatenation } of $u_1$ and $u_2$ define by
	$$u(t)=\left\{ \begin{array}{ll}
	u_1(t), & \;\;t\in [0, t_1]\\
	u_2(t-t_1), & \;\;t\in (t_1, t_1+t_2]
	\end{array}\right.$$

	The set of points {\it reachable from $x$ at time exactly} $\tau>0$, the set of
	points {\it reachable from $x$ up to time} $\tau>0$ and the {\it reachable set from} $x$ are respectively denoted by
	\[
	\AC_{\tau}(x):=\{\varphi(\tau, x, u), \; \;u\in \mathcal{U}\}, \;\;\; \AC_{\leq\tau}(x):=\bigcup
	_{t\in [0, \tau]}\AC_{t}(x)\; \; \; \mbox{ and }\; \; \; \AC(x):=\bigcup
	_{t>0}\AC_{t}(x).
	\]
	By $\mathcal{A}^*_{\tau}(x)$, $\mathcal{A}^*_{\leq \tau}(x)$ and $\mathcal{A}^*(x)$ we denote the
	corresponding sets for the time-reversed system. We say that the system
	(\ref{controlsystem}) is {\it locally accessible from} $x$ if $\mathrm{int}\AC_{\leq \tau}(x)$ and $\mathrm{int}\AC^*_{\leq \tau}(x)$ are nonempty for all $\tau>0$. The system is said to be {\it locally accessible} if it is locally accessible from any $x\in M$. A sufficient
	condition for local accessibility is the Lie algebra rank condition (LARC).
	It is satisfied if the Lie algebra $\mathcal{L}$ generated by the vector
	fields $x\in M\mapsto f_{u}(x):=f(x, u)$, for $u\in \Omega$, satisfies
	$\mathcal{L}(x)=T_{x}M$ for all $x\in M$.
	
	A subset $D\subset M$ is a {\it control set} of \ref{controlsystem} if it is
	maximal w.r.t. set inclusion with the following properties: 
	\begin{enumerate}
		\item[(i)] $D$ is {\it controlled invariant}, i.e., for each $x\in D$ there
		is $u\in \mathcal{U}$ with $\phi(\mathbb{R}_{+},x,u) \subset D$.
		
		\item[(ii)] {\it Approximate controllability} holds on $D$, i.e., $D
		\subset \operatorname{cl}\AC(x)$ for all $x\in D$.
	\end{enumerate}
	
	Following \cite{FCWK}, Proposition 3.2.4., any subset $D$ of $M$ with nonempty
	interior that is maximal with property (ii) in the above definition is a
	control set.
	
	The next result (see Lemma 3.2.13 of \cite{FCWK}) states the main properties of control sets with nonempty interior.
	
	\begin{lemma}
		\label{prop}
		Let $D$ be a control set of (\ref{controlsystem}) with nonempty interior. It holds:
		\begin{itemize}
			\item[(i)] If the system is locally accessible from all $x\in\cl D$, then $D$ is connected and $\cl\inner D=\cl D$;
			\item[(ii)] If $y\in\inner D$ is locally accessible, then $y\in\AC(x)$ for all $x\in D$;
			\item[(iii)] If the system is locally accessible from all $y\in\inner D$, then $\inner D\subset \AC(x)$ for all $x\in D$ and for every $y\in\inner D$ one has
			$$D=\cl\AC(y)\cap\AC^*(y).$$
			In particular, exact controllability holds on $\inner D$.
		\end{itemize}
	\end{lemma}
	
We say that a control set $D$ is {\it positively-invariant} (resp. {\it negatively-invariant}) it $\phi_{t, u}(D)\subset D$ for any $u\in\UC$ and $t>0$ (resp. $t<0$).

%	We say that $(u, x)\in\UC\times M$ is an {\it inner pair} if $\phi(\tau, x,  u)\in\inner\AC^+(x)$ for some $\tau>0$. The concept of inner pair is connected with the existence of control sets with nonempty interior (see Section 4.5 of \cite{FCWK}). As we will see, for linear control systems, if 

	\subsection{Semisimple theory}
	
	Standard references for the theory of semisimple Lie groups are Duistermat-Kolk-Varadarajan \cite{DKV}, Helgason \cite{Hel}, Knapp \cite{Knapp} and Warner \cite{War}. In the sequel, we only provide a brief review of the concepts used in this paper.%
	
	Let $G$ be a connected semisimple non-compact Lie group $G$ with finite center and Lie algebra $\fg$. We choose a Cartan involution $\zeta:\fg\rightarrow\fg$ and denote by $B_{\zeta}(X,Y) = -C(X,\zeta(Y))$ the associated inner product, where $C(X,Y) = \tr(\ad(X)\ad(Y))$ is the Cartan-Killing form. If $\fk$ and $\fs$ stand, respectively, for the eigenspaces of $\zeta$ associated with $1$ and $-1$, the {\it Cartan decompositions} of $\fg$ and $G$ are given, respectively, by%
	\begin{equation*}
	\fg = \fk \oplus \fs\;\;\;\;\mbox{ and }\;\;\;\;G = KS, \;\;\;\;\mbox{ where }\;\;\;K = \exp\fk \;\;\;\mbox{ and }\;\;\;S = \exp\fs.%
	\end{equation*}
	Fix a maximal abelian subspace $\fa\subset\fs$ and denote by $\Pi$ the set of roots for this choice. If $\fn:= \sum_{\alpha\in\Pi^{+}}\fg_{\alpha}$, where $\Pi^+$ is the set of positive roots and%
	\begin{equation*}
	\fg_{\alpha} = \left\{ X\in\fg\ :\ \ad(H)X = \alpha(H)X,\ \forall H \in \fa \right\}%
	\end{equation*}
	is the root space associated with $\alpha\in\Pi$, the {\it Iwasawa decompositions} of $\fg$ and $G$ are given, respectively, by%
	\begin{equation*}
	\fg = \fk \oplus \fa \oplus \fn \;\;\;\;\mbox{ and }\;\;\;\; G=KAN, \;\;\;\;\mbox{ where }\;\;N=\exp\fn \mbox{ and }A=\exp\fa.%
	\end{equation*}
	
	Let $\fa^+=\{H\in\fa; \;\alpha(H)>0, \;\alpha\in\Pi^+\}$ be the positive {\it Weyl chamber} associated with the above choices and consider $H\in\cl\fa^+$. The eigenspaces of $\ad(H)$ in $\fg$ are given by $\fg_{\alpha}$, $\alpha\in\Pi$ and $\fg_0=\ker\ad(H)$. The centralizer of $H$ in $\fg$ is given by
	\begin{equation*}
	\fz_H := \sum_{\alpha\in\Pi\cup\{0\}:\ \alpha(H)=0}\fg_{\alpha}%
	\end{equation*}
	and the centralizer in $\fk$ by $\fk_H := \fk \cap \fz_H$. They are, respectively, the Lie algebra of the centralizer of $H$ in $G$, $Z_H:=\{g\in G: \Ad(g)H=H\}$, and in $K$, $K_H=K\cap Z_H$. Since $G$ has finite center, the {\it centralizer} $M$ of $\fa$ in $K$ is a compact subgroup of $G$. This fact, together with the equality $Z_H=M(Z_H)_0$ implies that $Z_H$ has a finite number of connected components. The finite subgroup $\Gamma=Z_H/(Z_H)_0$ parametrizes the connected components of $Z_H$ and hence, $(Z_H)_{\gamma}$ will stand for the connected component of $Z_H$ related with $\gamma\in\Gamma$. 
	
	The negative and positive nilpotent subalgebras of type $H$ are given by%
	\begin{equation*}
	\fn_H := \sum_{\alpha\in\Pi :\ \alpha(H)>0}\fg_{\alpha} \mbox{\quad and\quad} \fn^-_H := \sum_{\alpha\in\Pi :\ \alpha(H)<0}\fg_{\alpha}.%
	\end{equation*}
		The parabolic subalgebra and the negative parabolic subalgebra of type $H$ are given, respectively, by 
		\begin{equation*}
			\fp_H := \sum_{\alpha\in\Pi\cup\{0\}:\ \alpha(H)\geq 0}\fg_{\alpha}\mbox{\quad and\quad}\fp^-_H := \sum_{\alpha\in \Pi\cup\{0\}:\ \alpha(H) \leq 0}\fg_{\alpha}.
		\end{equation*}
	At the group level, $N_H = \exp(\fn_H)$ and $N^-_H = \exp(\fn^-_H)$ stand for the connected nilpotent Lie subgroup associated with $\fn_H$ and $\fn_H^-$, respectively. The parabolic subgroups $P_H$ and $P_H^-$ are, respectively, the normalizer of $\fp_H$ and $\fp_H^-$ in $G$. It holds that $N_H$ is a normal subgroup of $P_H$ and the same is true for $N_H^-$ and $P_H^-$. In particular, it holds that
	$$P_H=Z_HN_H=N_HZ_H\;\;\;\mbox{ and }\;\;\;P^-_H=Z_HN_H^-=N_H^-Z_H.$$
	Moreover, the set 
	\begin{equation}
	\label{open}
	U_H:=P_HN^-_H=N_HZ_HN^-_H=N_HP^-_H
	\end{equation}
	is an open and dense subset of $G$.

\subsection{Linear vector fields on semisimple Lie groups}

	A vector field $\mathcal{X}$ on a connected Lie group $G$ is said to be {\it linear} if its flow 
	$(\varphi_t)_{t\in \R}$ is a $1$-parameter subgroup of $\mathrm{Aut}(G)$. Associated to any linear vector field $\mathcal{X}$ there is a derivation $\mathcal{D}$ of $\mathfrak{g}$ defined by the formula
	$$
	\mathcal{D}Y=-[\mathcal{X},Y](e),\mbox{ for all }Y\in \mathfrak{g}. 
	$$
	The relation between $\varphi_t$ and $\mathcal{D}$ is given by the
	formula 
	\begin{equation}
	(d\varphi_{t})_{e}=\mathrm{e}^{t\mathcal{D}}\; \; \; \mbox{ for all }\; \;
	\;t\in \mathbb{R}. \label{derivativeonorigin}
	\end{equation}%
	In particular, it holds that
	\begin{equation*}
	\varphi _{t}(\exp Y)=\exp (\mathrm{e}^{t\mathcal{D}}Y),\mbox{ for all }t\in 
	\mathbb{R},Y\in \mathfrak{g}.
	\end{equation*}
	
	Let $G$ be semisimple and consider $\XC$ to be a linear vector field on $G$. If $\DC$ stands for the derivation associated with $\XC$, the fact that $\fg$ is a semisimple Lie algebra implies that $\DC$ is inner, that is, there exists $X\in \fg$ such that $\DC=-\ad(X)$. 
	
	By equation (\ref{derivativeonorigin}) we get that
	$$\varphi_t(\exp Y)=\exp(\rme^{t\DC}Y)=\exp(\rme^{-t\ad(X)}Y)=C_{\rme^{tX}}(\exp Y)$$
	and since $G$ is connected, we conclude that $\varphi_t=C_{\rme^{tX}}$, where the minus sign on the above formula is connected with the choice of right-invariant vector fields.
	
	Following \cite{Hel} (Chapter 9, Lemma 3.1) the {\it Jordan} decomposition of an element $X\in\fg$ is the commuting decomposition  $X=E+H+N$ where $ H\in\cl(\fa^+)$, $E\in\fk_H$ and $\ad(N)$ is nilpotent. In particular, the Lie subalgebras $\fn_H$, $\fz_H$ and $\fn_H^-$ coincide, respectively, with the sum of the real generalized eigenspaces of $\DC$ associated with the eigenvalues with positive, zero and negative real parts. 
	
	We call the elements $E$, $H$ and $N$ obtained from the Jordan decomposition of $X$ the {\it elliptic, hyperbolic} and {\it nilpotent} parts of $\XC$, respectively. Moreover, the Jordan decomposition of $X$ implies that the flow of $\XC$ is given by the commutative product
	$$\varphi_t=C_{\rme^{tX}}=C_{\rme^{tE}}\circ C_{\rme^{tH}}\circ C_{\rme^{tN}}.$$ 
	A simple calculation shows that $N_H$, $N_H^-$ and $Z_H$ are $\varphi$-invariant. 
	
\subsection{Linear control systems on semisimple Lie groups}

A linear control system on a connected Lie group $G$ is a family of ordinary differential equations of the form
\begin{flalign*}
&&\dot{g}(t)=\XC(g(t))+\sum_{j=1}^mu_j(t)Y^j(g(t)), \; \; \;u=(u_1, \ldots, u_m)\in \UC &&\hspace{-1cm}\left(\Sigma_G\right)
\end{flalign*}
where $\XC$ is a linear vector field, $Y^{j}, \;j=1, \ldots, m$ are right-invariant vector fields and $u(t)\in\Omega$. The solutions of linear control systems are related to the flow of $\XC$ by the formula 
\begin{equation*}
\label{solutionform}
\phi(t, g, u)=\phi(t, e, u)\varphi_t(g)=L_{\phi(t, e, u )}(\varphi_{t}(g)).
\end{equation*}

Let us denote by $\AC_{\tau}, \AC^*_{\tau}, \AC_{\leq \tau}, \AC^*_{\leq \tau}, \mathcal{A}$ and $\AC^*$ the sets $\AC_{\tau}(e), \AC^*_{\tau}(e), \AC_{\leq \tau}(e), \AC^*(e)_{\leq \tau}, \mathcal{A}(e)$ and $\AC^*(e)$, respectively. 

The next proposition states the main properties of the reachable sets of
linear control systems (see for instance \cite{Jouan11}, Proposition 2).

\begin{proposition}
	\label{Prop2}
	With the previous notations it holds:
	
	\begin{itemize}
		\item[1.] $\AC_{\leq\tau}=\AC_{\tau}$;
		
		%\item[2.] $\mathcal{A}_{\tau }^*=(\varphi _{-\tau }(\mathcal{A}_{\tau }))^{-1}$
		
		\item[2.] $\mathcal{A}_{\tau }(g)=\mathcal{A}_{\tau
		}\varphi _{\tau }(g);$
		
		\item[3.] $\mathcal{A}_{\tau
			_{1}+\tau _{2}}=\mathcal{A}_{\tau _{1}}\varphi _{\tau _{1}}(\mathcal{A}
		_{\tau _{2}})=\mathcal{A}_{\tau _{2}}\varphi _{\tau _{2}}(\mathcal{A}_{\tau
			_{1}}).$
	\end{itemize}
\end{proposition}

The next result shows that the set $\mathcal{A}$ is invariant by right translations of elements whose $\varphi$-orbits are contained in $\mathcal{A}$ (\cite{DS}, Lemma 3.1).

\begin{lemma}
	\label{pointinvariance} Let $g\in \mathcal{A}$ and assume that $\OC(g, \varphi)\subset \mathcal{A}$. Then $\mathcal{A}\cdot
	g\subset \mathcal{A}$.
\end{lemma}

Let $G$ to be a connected semisimple Lie group with finite center. By using the notations introduced in Section 2.2 for the semisimple case, Theorem 3.9 of \cite{DSAy} gives us a strict relation between the subgroups $P_H$ and $P_H^-$ and the linear control system $\Sigma_G$ as follows.

\begin{theorem}
	\label{generalcase} Let $G$ be a connected semisimple Lie group with finite center. If $\mathcal{A}$ is open, then $(P_H)_0\subset \mathcal{A}$ and $(P^-_H)_0\subset \mathcal{A}^*$.
\end{theorem}

Next we extend the above results relating $\AC$ with the other connected components of $P_H$.

\begin{proposition}
	\label{improved}
	Let $G$ be a connected semisimple Lie group with finite center and assume that $\AC$ is open. For any $\gamma\in\Gamma$  we have that 
	$$(Z_H)_{\gamma}\cap\AC\neq\emptyset\;\;\implies\;\;(P_H)_{\gamma^n}\subset\AC, \;\;\mbox{ for any }\;\;n\in\Z.$$
\end{proposition}

\begin{proof}
	In fact, let $x\in (Z_H)_{\gamma}\cap\AC$. Since $(Z_H)_{\gamma}=x(Z_H)_0$ we get that any $z\in (P_H)_{\gamma}$ can be written as $z=xg$ with $g\in (P_H)_0$ and hence
	$$\OC(g, \varphi)\subset (P_H)_0\subset\AC\;\;\implies\;\; z=xg\in\AC\cdot g\subset\AC.$$
	Let us notice that $(P_H)_{\gamma^n}=((P_H)_{\gamma})^n$ and then, if $(P_H)_{\gamma^n}\subset\AC$ we have by the $\varphi$-invariance of $(P_H)_{\gamma^n}$ and Lemma \ref{pointinvariance} that $(P_H)_{\gamma^{n+1}}\subset\AC$. Since $\Gamma$ is finite, the above is true for any $n\in\Z$ showing the result.
\end{proof}

\bigskip

The next technical lemma will be useful ahead.

\begin{lemma}
	\label{technical}
	Let $\gamma\in\Gamma$ and $x\in (Z_H)_{\gamma}$.
	\begin{itemize}
		\item[(i)] It holds that $\AC(x)\subset\AC\cdot x$;
		\item[(ii)] If $\OC(x, \varphi)\subset\AC(x)$ then $\AC\cdot x\subset\AC(x)$.
	\end{itemize} 
\end{lemma}

\begin{proof}	(i) For any $z\in\AC(x)$ consider $\tau>0$ and $u\in\UC$ with $z=\phi(\tau, x, u)$. By the $\varphi$-invariance of $(Z_H)_{\gamma}$, there exists $l\in (Z_H)_0$ such that $\varphi_{\tau}(x)=lx$ and hence
	$$z=\phi(\tau, x, u)=\phi(\tau, e, u)\varphi_{\tau}(x)=\phi(\tau, e, u)lx\implies z\in \AC\cdot lx\subset \AC\cdot x,$$
	where for the inclusion we used that $\OC(l, \varphi)\subset (Z_H)_0\subset\AC$.
	
	(ii) Let $z\in\AC$ and write it as $z=\phi(\tau, e, u)$ for some $\tau>0$ and $u\in\UC$. Then
	$$zx=\phi(\tau, e, u)x=\phi(\tau, e, u)\varphi_{\tau}\left(\varphi_{-\tau}(x)\right)\in \phi_{\tau, u}(\OC(x, \varphi))\subset\phi_{\tau, u}(\AC(x))\subset\AC(x)$$
	and by the arbitrariness of $z\in\AC$ we get $\AC\cdot x\subset\AC(x)$ concluding the proof.  
\end{proof}

\begin{remark}
	It is important to notice that the assumption that $\AC$ is open is equivalent to the existence of some $\tau>0$ such that $e\in\inner\AC_{\tau}$ (see \cite{FCWK}, Lemma 4.5.2). In particular, if $\AC$ is open the system is locally accessible (see Theorem 3.3 of \cite{AyTi}) and $\AC^*$ is also open. There is an easily checkable algebraic condition that ensures the openness of $\AC$ called the {\it ad-rank} condition (see for instance \cite{Jouan11}, Proposition 6).  
\end{remark}

\section{Control sets of linear systems on semisimple Lie groups}

In this section will be assumed that $G$ is a connected semisimple Lie group with finite center and that $\Sigma_G$ is a linear control system such that $\AC$ is open. Let us denote by $H$ the hyperbolic part of the linear vector field $\XC$, drift of the system $\Sigma_G$.

By considering the homogeneous space $G/(Z_H)_0$ we have, by the $\varphi$-invariance of $(Z_H)_0$, a well defined system $\Sigma_{G/(Z_H)_0}$ on $G/(Z_H)_0$ induced by the linear control system $\Sigma_G$ (see \cite{JPh1}, Proposition 4). We aim to show that there is a strict relation between the control sets of $\Sigma_G$ and the ones of $\Sigma_{G/(Z_H)_0}$.

If $\pi:G\rightarrow G/(Z_H)_0$ is the canonical projection and $\Psi_t$ the flow induced by $\XC$ on $G/(Z_H)_0$ we have that
$$\Psi_t\circ\pi=\pi\circ\varphi_t.$$
Moreover, if $\LC_g$ stands for the left translation in $G/(Z_H)_0$ given by $x\in G/(Z_H)_0\mapsto gx\in G/(Z_H)_0$ we have that the solutions of $\Sigma_{G/(Z_H)_0}$ satisfy
\begin{equation}
\label{solutions}
\Phi(t, \pi(g), u)=\LC_{\phi(t, e, u)}(\Psi_t(\pi(g)))=\pi(\phi(t, g, u)).
\end{equation}
In particular, for any $g\in G$
$$\pi(\AC(g))=\AC_H(\pi(g))\;\;\mbox{ and }\;\pi(\AC^*(g))=\AC^*_H(\pi(g)),$$
are the reachable sets for from $\pi(g)$ for the induced system. By our assumption on the openness of $\AC$, there exists a control set of $\Sigma_G$ with nonempty interior containing the identity in its interior (see Corollary 4.5.11 of \cite{FCWK}). By Lemma \ref{prop} it is equal to $\cl(\AC)\cap\AC^*$ and is denoted by $\CC_1$ in the sequel.

\begin{theorem}
	\label{teo}
	The projection $\pi(\CC_1)$ is a control set for $\Sigma_{G/(Z_H)_0}$ and satisfies $\pi^{-1}(\pi(\CC_1))=\CC_1$.
\end{theorem}

\begin{proof}
	Since $\pi$ is an open map, it holds that $\pi(\CC_1)$ has nonempty interior. Moreover, by (\ref{solutions}) we have that 
	$$\mbox{ for all }\;g\in\CC_1, \;\;\;\pi(\CC_1)\subset\pi(\cl(\AC(g)))\subset\cl(\pi(\AC(g)))=\cl(\AC_H(\pi(g))),$$
	and therefore, $\pi(\CC_1)$ is contained in a control set $D_1$ for the control system $\Sigma_{G/(Z_H)_0}$. 
	
	The result is proved if we show that $\pi^{-1}(D_1)\subset\CC_1$. However, since $\CC_1\subset\pi^{-1}(D_1)$ it is enough to show that 
	\begin{equation}
	\label{relation}
	\pi^{-1}(D_1)\subset\cl(\AC(g)) \;\;\;\mbox{ for any }\;\;g\in\pi^{-1}(D_1).
	\end{equation}
	Let then $g_1, g_2\in\pi^{-1}(\inner D_1)$. Denote by $o=e\cdot (Z_H)_0$. Since exact controllability holds on $\inner D_1$ and $o\in\inner D_1$, there exist $t_1, t_2>0$ and $u_1, u_2\in\UC$ such that 
	$$\Phi(t_1, \pi(g_1), u_1)=o\;\;\mbox{ and }\;\;\Phi(t_2, o, u_2)=\pi(g_2)\iff\phi(t_1, g_1, u_1)= l_1\;\mbox{ and }\;\phi(t_2, l_2, u_2)=g_2$$
	for some $l_1, l_2\in (Z_H)_0$. On the other hand, the openness of $\AC$ implies $(Z_H)_0\subset\AC\cap\AC^*\subset\CC_1$ and hence there exist $t_3>0$ and $u_3\in \UC$ such that 
	$$\phi(t_3, l_1, u_3)=l_2\;\;\implies\;\;g_2=\phi(t, g_1, u), \;\;\mbox{ where }\;\;t=t_1+t_2+t_2>0\;\mbox{ and }\;u=(u_1*u_2)*u_3\in\UC.$$
	Therefore, $\pi^{-1}(\inner D_1)\subset \AC(g)$ for any $g\in\pi^{-1}(\inner D_1)$. Using the fact that $\inner D_1$ is dense in $D_1$ and by maximility of $\CC_1$ we get the desired result.
\end{proof}

\bigskip

We define now a group of homeomorphisms in $G/(Z_H)_0$. For any $\gamma\in\Gamma$ let us define the map $f_{\gamma}:G/(Z_H)_0\rightarrow G/(Z_H)_0$ by $g\cdot o\in G/(Z_H)_0\mapsto gl\cdot o\in G/(Z_H)_0$, where $l\in Z_H$ satisfies $\gamma=l\cdot o$. Since $(Z_H)_0$ is a normal subgroup of $Z_H$ the map $f_{\gamma}$ is a well defined homeomorphism of $G/(Z_H)_0$ satisfying:
\begin{itemize}
	\item[(i)] $(f_{\gamma})^{-1}=f_{\gamma^{-1}}$ for all $\gamma\in\Gamma$;
	\item[(ii)] $f_{\gamma_1\gamma_2}=f_{\gamma_2}\circ f_{\gamma_1}$ for $\gamma_1, \gamma_2\in\Gamma$;
\end{itemize}

\begin{lemma}
	\label{conjugation}
	For any $\gamma\in\Gamma$ it holds that 
	$$\Phi_{t, u}\circ f_{\gamma}=f_{\gamma}\circ\Phi_{t, u} \;\;\;\mbox{ for any }\;\;t\in\R, u\in\UC, \gamma\in \Gamma.$$
	In particular, $\Psi_t\circ f_{\gamma}=f_{\gamma}\circ\Psi_t$ for any $t\in\R$ and $\gamma\in \Gamma$.
\end{lemma}

\begin{proof} Since $\varphi_t((Z_H)_{\gamma})=(Z_H)_{\gamma}$ for any $\gamma\in \Gamma$ it holds that
$$\Phi(t, f_{\gamma}(g\cdot o), u)=\phi(t, e, u)\varphi_t(gl)\cdot o=\phi(t, e, u)\varphi_t(g)\varphi_t(l)\cdot o$$
$$=\phi(t, e, u)\varphi_t(g)l\cdot o=f_{\gamma}(\phi(t, e, u)\varphi(g)\cdot o)=f_{\gamma}(\Phi(t, g\cdot o, u))$$
and the result follows.
\end{proof}

A direct consequence of Lemma \ref{conjugation} is that $D_{\gamma}:=f_{\gamma}(D_1)$ is a control set with nonempty interior of $\Sigma_{G/(Z_H)_0}$, where $D_1=\pi(\CC_1)$ is the projection of the control set $\CC_1$ of $\Sigma_G$ as proved in Theorem \ref{teo}. The set 
$$\Gamma_0:=\{\gamma\in\Gamma; \;D_{\gamma}=D_1\}$$ 
is a subgroup of $\Gamma$ and the map $\xi$ given by $\Gamma_0\gamma \in\Gamma_0\setminus\Gamma\mapsto D_{\gamma}$ is a well-defined injective map, since  
	$$\Gamma_0\gamma_1=\Gamma_0\gamma_2\iff \gamma_2\gamma_1^{-1}\in\Gamma_0\iff D_1=D_{\gamma_2\gamma_1^{-1}}\iff D_{\gamma_1}=D_{\gamma_2}.$$

The next result shows that the control sets of a linear control system are related with the control sets $D_{\gamma}$ for $\gamma\in\Gamma$.

\begin{lemma}[Fundamental Lemma]
	\label{controlquotient}
	If $\CC$ is a control set of $\Sigma_G$ with nonempty interior then
	$$\CC\subset \pi^{-1}(D_{\gamma}), \;\mbox{ for some }\;\gamma\in\Gamma.$$
\end{lemma}

\begin{proof}
	Let $\CC\subset G$ be a control set with nonempty interior of $\Sigma_G$. By equation (\ref{open}) the set 
	$$U_H=\dot{\bigcup_{\gamma\in \Gamma}}(P_H)_{\gamma}N_H^-$$
	is an open and dense subset of $G$ and hence $\inner \CC\cap (P_H)_{\gamma_1}N^-_H\neq\emptyset$ for some $\gamma_1\in\Gamma$. There exist
	$g\in (P_H)_{\gamma_1}$, $h\in N_H^-$ with $gh\in \inner \CC$. Since in $\mathrm{int}\CC$ we have exact controllabillity, there exists $\tau>0$ and $u\in\UC$ such that $$\phi(n\tau, gh, u)=gh,\;\;\mbox{ for each }n>0.$$
	If $\varrho$ stands for a left invariant Riemannian metric on $G$ we get
	$$\varrho(gh, \phi(n\tau, g, u))=\varrho(\phi(n\tau, gh, u), \phi(n\tau, g, u))=\varrho(\varphi_{n\tau}(h), e).$$
	Since $N_H^-=\exp\fn^-_H$ and $\ad(X)|_{\fn_H^-}$ has only eingevalues with negative real part, we have that $\varphi_{n\tau}(h)\rightarrow e$ as $n\rightarrow+\infty.$
	On the other hand, the fact that $g\in (P_H)_{\gamma_1}$ implies that $\phi(n\tau, g, u)\in \AC\cdot l_1$ for each $n>0$, where $l_1\in (Z_H)_{\gamma_1}$. Hence, $gh\in\cl(\AC\cdot l_1)$ implying that $\inner \CC\cap \AC\cdot l_1\neq\emptyset$. Let then $x\in\inner \CC\cap\AC\cdot l_1$ and $y\in\inner\CC$. By exact controllability there exists $t>0$ and $u\in\UC$ with $y=\phi_{t, u}(x)$. If we write $\varphi_t(l_1)=ml_1$ with $m\in (Z_H)_{\gamma_1}$ we get
	$$ y=\phi_{t, u}(x)\in\phi_{t, u}\left(\AC\cdot l_1\right)=\phi_{t, u}\left(\AC\right)\cdot\varphi_{t}(l_1)\subset \AC\cdot ml_1\subset\AC\cdot l_1$$
	where for the last inclusion we used Lemma \ref{pointinvariance}. By the arbitrariness of $y\in\inner \CC$ we conclude that
	$$\inner \CC\subset\AC\cdot l_1.$$
	By arguing analogously for $U_H^{-1}$ we assure the existence of $\gamma_2\in\Gamma$ such that $\inner \CC\subset \AC^*\cdot l_2$ where $l_2\in (Z_H){\gamma_2}$. Therefore, 
	\begin{equation}
	\label{1}
	\inner \CC\subset\AC\cdot l_1\cap \AC^*\cdot l_2, \;\;\mbox{ where }\;\;l_i\in (Z_H)_{\gamma_ i}, \;i=1, 2.
	\end{equation} 	
	In particular, $\AC\cdot l_1\cap \AC^*\cdot l_2\neq\emptyset$. Let then $x\in \AC\cdot l_1\cap \AC^*\cdot l_2$ and consider $t_1, t_2>0$ and $u_1, u_2\in\UC$ such that 
	$$x=\phi_{t_1, u_1}(e)\cdot l_1\;\;\mbox{ and }\;\; \phi_{t_2, u_2}(xl_2^{-1})=e\;\;\implies e=\phi_{t_2, u_2}(\phi_{t_1, u_1}(e)l_1l_2^{-1})=\phi_{t, u}(e)\varphi_{t_2}(l_1l_2^{-1}).$$
	where $t=t_1+t_2$ and $u=u_1*u_2$. Therefore, $\varphi_{t_2}(l_2l^{-1}_1)=\phi_{t, u}(e)\in\AC$ implying that $(Z_H)_{\gamma_2\gamma_1^{-1}}\cap\AC\neq\emptyset$. By Proposition \ref{improved} we get  
	\begin{equation}
	\label{2}
	(P_H)_{(\gamma_1\gamma_2^{-1})}=(P_H)_{(\gamma_2\gamma_1^{-1})^{-1}}\subset\AC\;\;\implies \;\;(Z_H)_{\gamma_1}\cap \AC\cdot l_2\neq\emptyset.
	\end{equation}
	Since, $\pi(\AC\cdot l_i)=\AC_H(\gamma_i)$ and $\pi(\AC^*\cdot l_i)=\AC_H^*(\gamma_i)$, equation (\ref{2}) implies 
	$$\gamma_1\in\pi(\AC\cdot l_2)=\AC_H(\gamma_2)\;\;\implies\;\;\AC_H(\gamma_1)\subset\AC_H(\gamma_2)$$
	and using equation (\ref{1}) we get 
	$$\pi(\inner \CC)\subset\pi(\AC\cdot l_1)\cap\pi(\AC^*\cdot l_2)\subset \AC_H(\gamma_1)\cap\AC_H^*(\gamma_2)\subset \AC_H(\gamma_2)\cap\AC_H^*(\gamma_2)\subset D_{\gamma_2}$$ 
	 concluding the proof.	
\end{proof}

Now we can prove our main result.

\begin{theorem}
	If $\CC$ is a control set with nonempty interior of $\Sigma_G$ then 
	$$\CC\subset R_l(\CC_1), \;\;\mbox{ for some }l\in Z_H.$$
\end{theorem}

\begin{proof}
	In fact, for any $l\in (Z_{H})_{\gamma}$, it holds that $\pi\circ R_l=f_{\gamma}\circ\pi$ and consequently
	$$\pi^{-1}(D_{\gamma})=R_l(\CC_1).$$
	In fact, 
	$$x\in \pi^{-1}(D_{\gamma})\iff\pi(x)\in D_{\gamma}=f_{\gamma}(D_1)\iff f_{\gamma^{-1}}(\pi(x))\in D_1\iff \pi(R_{l^{-1}}(x))\in D_1=\pi(\CC_1)$$
	$$\iff R_{l^{-1}}(x)\in \pi^{-1}(\pi(\CC_1))=\CC_1\iff x\in R_l(\CC_1)$$
	and the result follows from Lemma \ref{controlquotient}.
\end{proof}

The next result shows that for linear control system whose drift has trivial nilpotent part the right translations of $\CC_1$ coincides with the control sets of $\Sigma_{G}$. This case is particularly important because generically linear vector fields have trivial nilpotent part since it is true as soon $\alpha(H)\neq 0$ for $\alpha\in\Pi$. 

\begin{theorem}
	\label{hyper}
	If $\AC$ is open and $\XC$ has trivial nilpotent part, then 
	$$R_l(\CC_1)=\pi^{-1}(D_{\gamma})$$
	are all the control sets with nonempty interior of $\Sigma_G$. In this case $\Sigma_G$ admits exactly $|\Gamma|/|\Gamma_0|$ control sets with nonempty interior.
\end{theorem}

\begin{proof}
	Since we already have that any control set with nonempty interior is contained in $\pi^{-1}(D_{\gamma})$ for some $\gamma\in \Gamma$ it is enough to show that, if the nilpotent part of $\XC$ is trivial, we actually have that $\pi^{-1}(D_{\gamma})$ is a control set of $\Sigma_G$ for any $\gamma\in\Gamma$.
	
	By the assumption on the linear vector field we have that the flow of the linear vector field restricted to $Z_H$ is given by $\varphi_t|_{Z_H}=C_{\rme^{tE}}$. Since $E\in\fk_H$ and $K$ is compact, we get that $\{\rme^{tE}, \;t\in\R\}\subset K$ is bounded and hence, if $x\in Z_H$ we have that $\cl(\OC(x, \varphi))$ is a compact subset.  Therefore, for any $\gamma\in \Gamma$ and any $x\in (Z_H)_{\gamma}$ we obtain by Corollary 4.5.11 of \cite{FCWK} that there exists a control set $\CC_x$ with nonempty interior such that 
	$$\cl(\OC(x, \varphi))\subset\inner \CC_x.$$
	In particular $\CC_x=\cl(\AC(x))\cap\AC^*(x)$ but $\OC(x, \varphi)\subset\AC(x)$ implies by item (ii) of Lemma \ref{technical} that $\AC(x)=\AC\cdot x$ for any $x\in (Z_H)_{\gamma}$. On the other hand, for any $x, y\in (Z_H)_{\gamma}$ there is $z\in (Z_H)_0$ with $zx=y$ which gives us
	$$\AC(x)=\AC\cdot x\subset \AC\cdot zx=\AC\cdot y=\AC(y)$$
	and hence $\AC(x)=\AC(y)$. Analogously, for any $x, y\in (Z_H)_{\gamma}$ we get $\AC^*(x)=\AC^*(y)$ implying that $\CC_{x}=\CC_y$.  Moreover, 
	$$\inner D_{\gamma}=\AC_{H}(\gamma)\cap \AC^*_{H}(\gamma)\implies \pi^{-1}(\inner D_{\gamma})=\bigcup_{x, y\in (Z_H)_{\gamma}}\AC\cdot x\cap\AC\cdot y=\AC(x)\cap\AC^*(x)\subset \inner \CC_x.$$
	Using that $\pi(\CC_x)\subset D_{\gamma}$ and that $\inner D_{\gamma}$ is dense in $D_{\gamma}$ we have $\CC_x=\pi^{-1}(D_{\gamma})$ as stated.	
\end{proof}

\begin{remark}
	By following the idea of the proof of the above theorem, it is not hard to show that $\pi^{-1}(D_{\gamma})$ is a control set as soon as $(Z_H)_{\gamma}$ possesses a fixed or a periodic point for the flow of $\XC$, even when the nilpotent part of $\XC$ is not trivial. 
\end{remark}

We finish this section by showing that the existence of an invariant control set is equivalent to the controllability of $\Sigma_G$.

\begin{theorem}
	The only possible positively-invariant (resp. negatively-invariant) control set of $\Sigma_G$ with nonempty interior is $G$.
\end{theorem}

\begin{proof} Our proof is divided in two steps:
	
	\medskip
	
	{\bf Step 1:} $\mathcal{A}=G$ if and only if $\mathcal{A}^*=G$
	
	Since both cases are analogous we will only show that $\mathcal{A}=G\;\implies\;\mathcal{A}^*=G$. Recall that $G$ being semisimple the derivation $\mathcal{D}=-\ad(\mathcal{X})$ is inner and equal to $-\ad(X)$ for some right-invariant
		vector field $X$. We can consequently define an invariant system $\Sigma_I$ by:
		\begin{flalign*}
		&&\dot{g}(t)=X(g(t))+\sum_{j=1}^mu_j(t)Y^j(g(t)), \; \; \;u=(u_1, \ldots, u_m)\in \UC. &&\hspace{-1cm}\left(\Sigma_I\right)
		\end{flalign*}
		
		The solutions $\phi^I$ of $\Sigma_I$ and $\phi$ of $\Sigma_G$ are related by 
		\begin{equation}
		\label{bla}
		\phi^I(t, g, u)=R_{\rme^{tX}}(\phi(t, g, u)), \;\;\mbox{ for any }t\in\R, g\in G, u\in\UC.
		\end{equation}
		An easy proof of equation \ref{bla} can be found in \cite{Jouan11}, Proposition 8. As consequence the reachable set at time $t\geq 0$ from the identity for
		$\Sigma_I$ is $\mathcal{S}_t=\mathcal{A}_t\exp(tX)$.
		
		The reachable set from the identity $\mathcal{S}=\bigcup_{t\geq 0}\mathcal{S}_t$ is
		a semigroup with nonempty interior. It is said to be left reversible (resp. right reversible) if
		$\mathcal{S}\mathcal{S}^{-1}=G$ (resp. $\mathcal{S}^{-1}\mathcal{S}=G$). Following Theorem 6.7 of \cite{ST95}, if $G$ is a connected semisimple Lie group with finite center then $G$ itself is the only subsemigroup with nonempty interior which is left or right reversible.
		
		Assume then that $\mathcal{A}=G$. The first thing to show is that $\mathcal{S}=G$. Let $g\in G$.
		Since $\mathcal{A}=G$ there exists $t\geq 0$ such that $g\in
		\mathcal{A}_t=\mathcal{S}_t\rme^{-tX}$. This implies $\mathcal{S}\exp(-\R_+X)=G$. However, since $\exp(-\mathbb R_+X)\subset\mathcal{S}^{-1}$  we obtain
		$\mathcal{S}\mathcal{S}^{-1}=G$ and hence $\mathcal{S}=G$.
		
		We can now prove that $\mathcal{A}^*=G$. Let $g\in G$. There exists $t\geq 0$ such
		that $g^{-1}\in \mathcal{S}_t$ or equivalently $g^{-1}\rme^{-tX}\in\mathcal{A}_t$. Consequently:
		$$
		\rme^{-tX} =g^{-1}\rme^{-tX}\rme^{tX}g\rme^{-tX}=g^{-1}\rme^{-tX}\varphi_t(g)\in \mathcal{A}_t\varphi_t(g)=\mathcal{A}_t(g).
		$$
		But $\mathcal{A}=G$ and there exists $s> 0$ and $u\in\UC$ such that
		$\exp(tX)=\phi_{s, u}(e)$, so that:
		$$
		e=\rme^{-tX}\varphi_s(\rme^{-tX})=\phi_{s, u}(e)\varphi_s(\rme^{-tX})=\phi_{s, u}(\rme^{-tX})\in \phi_{s, u}(\AC^+(g))\subset\AC(g).
		$$
		and hence $g\in\AC^*$ concluding the proof.
		
		{\bf Step 2:} If $\Sigma_G$ admits a positively-invariant (resp. negatively-invariant) control set $\CC$ with nonempty interior then $\CC=G$. 
	 
	 In fact, let us assume that $\CC$ is a positively-invariant control set of $\Sigma_G$ with nonempty interior. By Theorem \ref{controlquotient} we get that $\pi(\CC)\subset D_{\gamma}$ for some $\gamma\in\Gamma$. Since $\pi(\inner \CC)\subset \inner D_{\gamma}$ and exact controllability holds on $\inner D_{\gamma}$, we can always build an periodic orbit passing for a given point in $\inner D_{\gamma}$ and intersecting $\pi(\inner \CC)$ which by the positively-invariance of $\CC$ implies that $\inner D_{\gamma}\subset\pi(\CC)$ and hence that $D_{\gamma}$ is positively-invariant. Since $D_1=f_{\gamma^{-1}}(D_{\gamma})$ we get that $D_1$ is also positively-invariant and by Theorem \ref{teo}, the same holds for $\CC_1$. In particular, $\CC_1$ is closed which by Theorem 3.6 of \cite{GzDaAv} implies that $\AC^*=G$ and by the previous step that $\AC=G$ which implies the result.
\end{proof}

\begin{remark}
	It is important to remark that Step 1 on the previous proof was first stated in \cite{Jouan11} but for unrestricted
	inputs. Moreover, the idea of the ``reversible semigroups" trick comes from \cite{SA01}.
\end{remark}

\section{Example}

Let $G=\mathrm{Sl}(2)$ be the three-dimensional semisimple Lie group of the $2\times 2$ matrices with determinant equal to one. Its Lie algebra is given by $\fg=\mathrm{sl}(2)$, the set of the $2\times 2$ matrices with zero trace.

Denote any element $h\in G$ by $h=(v_1, v_2)$, where $v_1, v_2\in\R^2$ satisfies $\langle v^*_1, v_2\rangle=1$. Here $v_1^*$ is the orthogonal vector of $v_1$ obtained by a counter-clockwise rotation of $\pi/2$. If $A:=\{g\in G; \;g\;\mbox{ is diagonal with positive entries}\}$ we have the diffeormorphism
$$\psi: G/A\rightarrow S^1\times \R\;\;\mbox{ defined by }\;\;\psi(v_1, v_2)=\left(\frac{v_1}{|v_1|}, \langle v_1, v_2\rangle\right).%\footnote{\color{red} Follows from the fact that $\langle v_1^*, v_2\rangle=1$ and $\langle v_1, v_2\rangle=x$ admitis a unique solution in $v_2$!!}
$$
Note that  
$$\psi (g(v_1, v_2))=\pi(gv_1, gv_2)= \left(\frac{gv_1}{|gv_1|}, \langle gv_1, gv_2\rangle\right), \;\;\mbox{ for any }\;\;g\in G$$
and hence, the flow of any right-invariant vector field $\rme^{tX}$, $X\in\fg$, passes to $S^1\times \R$ as 
\begin{equation}
\label{u}
(t, (v, x))\in\R\times (S^1\times \R)\mapsto \left(\frac{\rme^{tX}v}{|\rme^{tX}v|}, \langle \rme^{tX}v, \rme^{tX}(xv+v^*)\rangle\right)
\end{equation}
where we used that $\psi(v, xv+v^*)=(v, x)$. 

A simple calculation shows that the flow (\ref{u}) is associated with the vector field
$$f_X(v, x)=\Bigl(Xv-\langle Xv, v)v, \langle Xv, xv+v^*\rangle+\langle v, X(xv+v^*)\Bigr).$$
We are interested in the dynamical behaviour of the flow of $f_H$ where $H\in\fg$ has a pair of distinct real eigenvalues. For such case, there is a basis $\{v_H, v_H^*\}$ of eigenvalues of $H$ that we always assume ordered such that $v_H$ is associated with the positive eigenvalue. 

For such case, the first component of the solution is given by 
$$\phi_1(t, (v, x))=\frac{\rme^{tH}v}{|\rme^{tH}v|}$$  
and its dynamical behaviour in the circle is given as in the picture ahead (see Figure \ref{a}).

The second component of the solution of $f_H$ is then 
$$\phi_2(t, (v, x))=\cos^2\theta_H(v)\rme^{2t\lambda_H}(x-\tan\theta_H(v))+\sin^2\theta_H(v)\rme^{-2t\lambda_H}(x+\cot\theta_H(v))$$
where $\lambda_H$ is the positive eigenvalue of $H$ and $\theta_H(v)$ is defined by $\cos\theta_H(v)=\langle v_H, v\rangle_H$ where $\langle\cdot, \cdot\rangle_H$ is the inner product that makes $\{v_H, v_H^*\}$ orthogonal. The dynamical behaviour of $\phi_2$ is given by the right Figure \ref{a}.
\begin{figure}[h!]
	\begin{center}
		\includegraphics[scale=.3]{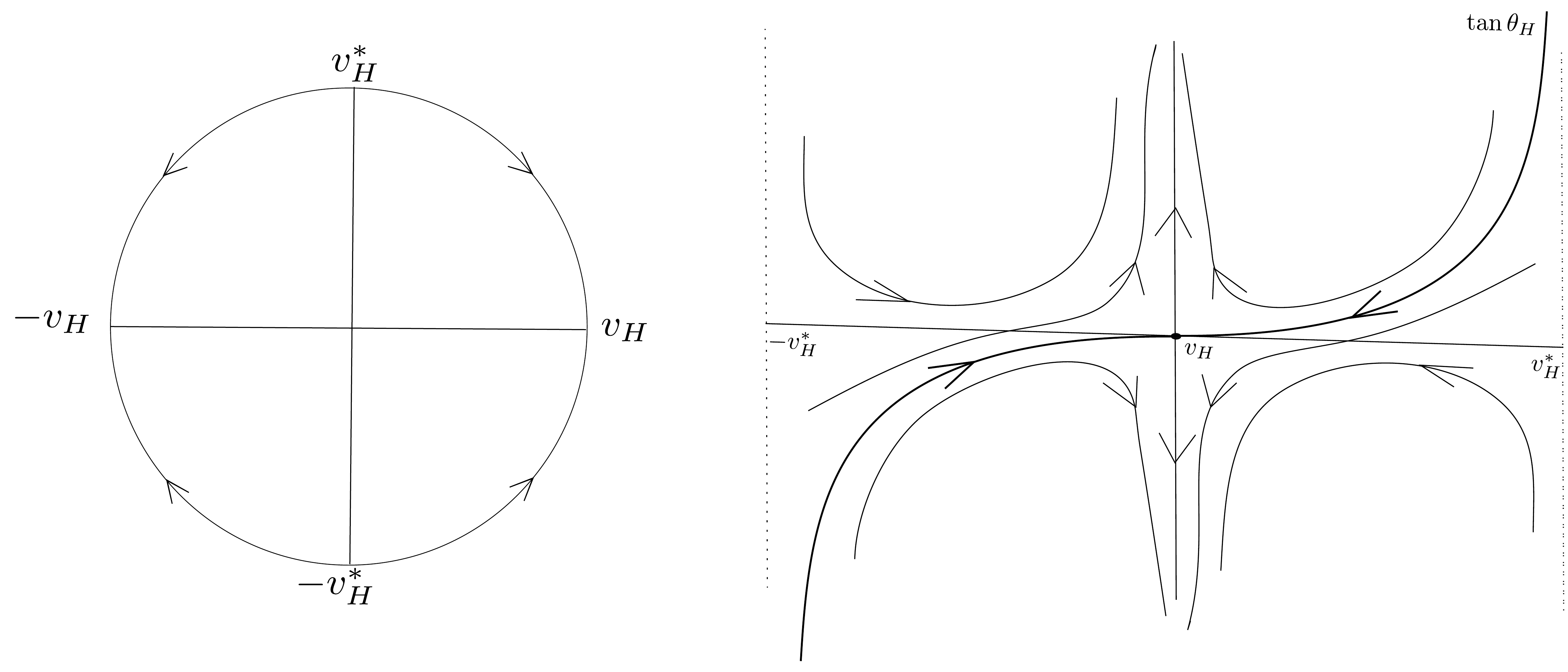}
	\end{center}
	\par
	\caption{Dynamical behaviour of $\phi_1$ and $\phi_2$.}
	\label{a}
\end{figure}

Let us consider now a linear control system $\Sigma_G$ given by $\dot{g}=\XC(g)+uX$ with $u\in [-\rho, \rho], \;\rho >0$ and
$\DC=-\ad(H)$ the associated derivation. Assume that
\begin{itemize}
	\item[1.] $H$ is a nonzero diagonal matrix;
	\item[2.] $H_u:=H+uX$ has a pair of distinct real eigenvalues for any $u\in [-\rho, \rho]$;
	\item[3.] $\{X, [H, X], [H, [H, X]]\}$ is a basis for $\fg$ \footnote{A pair of matrices satisfying the conditions are, for instance, $H=\left(\begin{array}{cc} 1 & 0 \\ 0 & -1\end{array}\right)\;\;\mbox{ and }\;\;X=\left(\begin{array}{cc} 1 & 1 \\ 1/2 & -1\end{array}\right)$.}.
\end{itemize}	

The above conditions imply that the system is not controllable (see \cite{AySM}), that $\AC$ is open and that $Z_H=\pm A$. Moreover, the fact that $\rme^{tH}\in A, \;t\in\R$ implies that the induced control on $G/A$ coincides with the associated invariant system. Moreover, the fact that the piecewise constant control functions are dense in $\UC$, implies that we only need to analyze how the concatenations of $\rme^{tH_u}$ for $u\in [-\rho, \rho]$ acts on $S^1\times \R$. 	

Following \cite{FCWK}, Chapter 6, the system on $S^1$ given by $\phi_1$ has, for small $\rho>0$, four control sets where the one containing $e_1$ is the closed interval on $S^1$ given by $[v_{-\rho}, v_{\rho}]$, where $v_u:=v_{H_u}$, $u\in[-\rho, \rho]$ is the attractor of $\rme^{tH_u}$. The control set $D_1$ that contains $\pi(e)=e_1$ is then given in Figure \ref{b}. Moreover, since $Z_H/(Z_H)_0=\{\pm 1\}$ implies that $f_{-1}(v, x)=(-v, x)$ and hence $D_{-1=}f_{-1}(D_1)$ (see Figure \ref{b}).

\begin{figure}[h]
	\begin{center}
		\includegraphics[scale=.4]{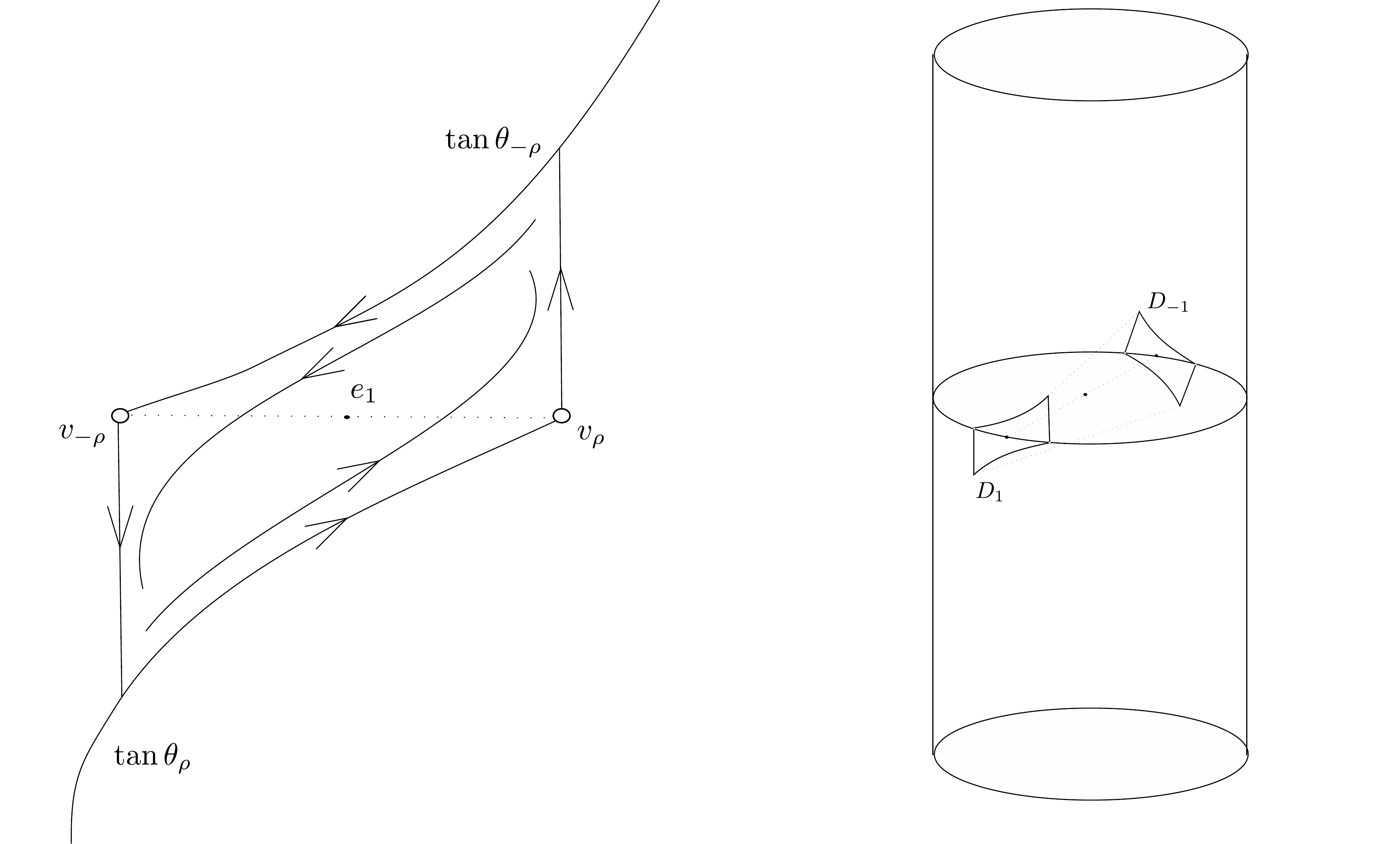}
	\end{center}
	\par
	\caption{The control sets of $\Sigma_{G/(Z_H)_0}$ associated with $\Gamma=\{\pm1\}$.}
	\label{b} 
\end{figure}

By Theorem \ref{hyper} we have that the control sets with nonempty interior of $\Sigma_G$ are given by $\pi^{-1}(D_1)$ and $\pi^{-1}(D_{-1})$.

\begin{remark}
	It is not hard to see that there are control sets around the points $(e_2, 0)$ and $(-e_2, 0)$ who are still related by the map $f_{-1}$. That shows that the induced control system on $\Sigma_{G/(Z_H)_0}$ can have more control sets than the ones given by $D_{\gamma}$, $\gamma\in\Gamma$.
\end{remark}

%\section*{Acknowledgements}

%The first author was supported by Proyecto Fondecyt n$^{o}$ 1150292, Conicyt,
%Chile and the second one was supported by Fapesp grants n$^{o}$ 2016/11135-2 and 2018/10696-6.{\color{red} Felipe and Guilherme!!}

\end{document}